\documentclass[a4paper,12pt]{amsart}

\usepackage{amsmath}
\usepackage{amssymb,amsbsy,amsmath,amsfonts,amssymb,amscd}
\usepackage{latexsym}
\usepackage{amsthm}
\usepackage{graphics}
\usepackage{color}
\usepackage{tikz}
 \usepackage{tikz-3dplot}
\usetikzlibrary{arrows}
\usepackage{tikz-cd}
\usepackage{pdfpages} %cover
\usepackage{longtable}
\usepackage{xy}
\usepackage{array}
\usepackage{color}
\usepackage{comment}

\input xy
\xyoption{all}

\newcommand\sC{{\mathcal C}}

\newcommand\sL{{\mathcal L}}

\newcommand\sB{{\mathcal B}}

\newcommand\sN{{\mathcal N}}

\newcommand\sM{{\mathcal M}}

\usepackage{textcomp} %The package supports the Text Companion fonts, which provide many text symbols (such as baht, bullet, copyright, musicalnote, onequarter, section, and yen), in the TS1 encoding.
               
\newcommand\la{\lambda}

\newcommand\be{\beta}
\newcommand\e{\epsilon}
\newcommand\s{\sigma}

\newcommand\Ga{\Gamma}

\newcommand\de{\delta}

\newcommand{\NN}{\ensuremath{\mathbb{N}}}
\newcommand{\hol}{\ensuremath{\mathcal{O}}}

\newcommand{\PP}{\ensuremath{\mathbb{P}}}

\newcommand{\ra}{\ensuremath{\rightarrow}}

\newcommand{\F}{\ensuremath{\mathbb{F}}}

\def\eea{\end{eqnarray*}}
\def\bea{\begin{eqnarray*}}

\newcommand\dual{\mathrel{\raise3pt\hbox{$\underline{\mathrm{\thinspace d
\thinspace}}$}}}
\newcommand\qe{\ifhmode\unskip\nobreak\fi\quad $\Box$}       % box for QED

\def\BOX{\hfill\lower.5\baselineskip\hbox{$\Box$}}
\newtheorem{theorem}{Theorem}
\newtheorem{theo}[theorem]{Theorem}
\newtheorem{rem}[theorem]{Remark}

\newtheorem{prop}[theorem]{Proposition}
\newtheorem{cor}[theorem]{Corollary}
\newtheorem{lemma}[theorem]{Lemma}
\newtheorem{example}[theorem]{Example}

\theoremstyle{definition}
\newtheorem{defin}[theorem]{Definition}

\newcommand{\sR}{\ensuremath{\mathcal{R}}}

\newenvironment{dedication}
        {\begin{quotation}\begin{center}\begin{em}}
        {\par\end{em}\end{center}\end{quotation}}

\usepackage{hyperref}

\makeatletter
\def\tagform@#1{\maketag@@@{\ignorespaces#1\unskip\@@italiccorr}}
\makeatother

\newcolumntype{H}{@{}>{\lrbox0}l<{\endlrbox}} %Spalten ignorieren

\begin{document}

\title[Quartic surfaces, singularities]{The number of singular points of  quartic  surfaces (char=2).}
\author{Fabrizio Catanese}
\address{Lehrstuhl Mathematik VIII, 
 Mathematisches Institut der Universit\"{a}t
Bayreuth, NW II\\ Universit\"{a}tsstr. 30,
95447 Bayreuth, Germany \\ and Korea Institute for Advanced Study, Hoegiro 87, Seoul, 
133--722.}
\email{Fabrizio.Catanese@uni-bayreuth.de}

\thanks{
The author acknowledges support of the ERC 2013 Advanced Research Grant - 340258 - TADMICAMT}

\maketitle

\begin{dedication}
Dedicated to Bernd Sturmfels on the occasion of his 60-th  
 birthday.
\end{dedication}

\begin{abstract}
We show that the maximal number of singular
points of a normal quartic surface $X \subset \PP^3_K$ defined over an algebraically closed field $K$ of characteristic $2$ is at most   $20$,
and that if equality is attained, then  the  minimal resolution of $X$ is a supersingular 
 K3 surface and the singular points are $20$ nodes.

We produce examples with 14 nodes. In a sequel to this paper (in two parts, the second in collaboration with Matthias Sch\"utt)
we show that the optimal bound is indeed 14, and that if equality is attained, then  the  minimal resolution of $X$ is a supersingular 
 K3 surface and the singular points are $14$ nodes.

We also obtain some smaller upper bounds under several geometric assumptions holding 
at  one of the   singular points $P$ (structure of  tangent cone, separability/inseparability of the projection
with centre $P$).

\end{abstract}

\tableofcontents

\setcounter{section}{0}

\section{Introduction} 

Given an irreducible surface $X$ of degree $d$ in $\PP^3$, defined over an algebraically closed  field $K$, which is normal, that is, with only   finitely many singular points, one important question is to determine 
 the maximal number $\mu(d)$  of singular points that $X$ can have (observe however, see for instance \cite{bsegre}, \cite{angers}, \cite{j-r}, \cite{p-t}, \cite{duco},
 that the research has been more focused on the seemingly simpler question of finding the maximal number of nodes, that is, ordinary quadratic singularities). The case of $d=1,2$ being trivial ($\mu (1)= 0, \mu (2)=1$), the first interesting cases are for $d= 3,4$. 
 
 We have that $\mu (3)= 4$,  while $ \mu (4)=16$  if char(K) $ \neq 2$.

\bigskip

For $d=3$ (see  proposition \ref{monoid} and for instance \cite{kummers} for  more of classical references), a normal cubic surface $X$ can have at most 4 singular points, no three of them can be collinear, 
and if it does have 4 singular points, these are linearly independent, hence $X$  is projectively equivalent to the 
so-called Cayley cubic, first apparently found by Schl\"afli, see \cite{schlaefli}, \cite{cremona},  \cite{cayley}
(and then the singular points are nodes).

The Cayley cubic has the simple equation

$$X : = \{ x : = (x_0, x_1, x_2, x_3) | \s_3(x ): =  \sum_i \frac{1}{x_i} x_0 x_1 x_2 x_3 = 0\} $$ 
Here $\s_3$ is the third elementary symmetric function (the four singular points are the 4 coordinate points).

\bigskip
 
  The main purpose of this paper  is to show (Theorem \ref{mt}) that, if $ char(K) = 2$, then  $\mu(4)\leq  20$.
  
  \bigskip
  
  But, whereas  symmetric functions produce surfaces with the maximal number of singularities for degree $d=3$,
or for $d=2$ in characteristic $\neq 2$ (see for instance \cite{kummers}, \cite{m-o2}) we show 
in the last section  that for $d=4$ and 
char $=2$ a special (codimension 1) subfamily of the family of symmetric functions produce quartics with  at most $10$ singular points.

In this paper we also produce the following simple example, of  quartic surfaces with $14$ singular points,
and producing what we call  `the inseparable case' (for which we show that the maximum number of singular points is $\leq 16$):
$$X : = \{  (z, x_1, x_2, x_3) |  z^2 (x_1 x_2 + x_3^2) + (y_3 + x_1)(y_3 + x_2) y_3 (y_3 + x_1+ x_2)=0 ,$$
$$ y_3 : =  a_3 x_3 + a_1 x_1 + a_2 x_2, \  a_3 \neq 0, a_1, a_2, a_3  \ {\rm  general} \} .$$

A normal  quartic surface can have,   if char(K) $ \neq 2$, at most 16 singular points.
Indeed,  if char(K) $ \neq 2$, and $X$ is a normal quartic surface, by proposition \ref{monoid} it has at most $7$ singular points if it has a triple point,
else  it suffices to project from a double  point of the quartic  to the plane, and to use the bound for the number of singular points for a plane curve of degree $6$, which equals $15$, to establish that $X$ has at most $16$ singular points. 

Quartics with $16$ singular points  (char(K) $ \neq 2$) have necessarily nodes as singularities, and they are the so called 
 Kummer surfaces \cite{kummer} (the first examples were found by Fresnel, 1822).

There is a long history of research on Kummer quartic surfaces in char(K) $ \neq 2$, for instance it is  
 well known  that if $d=4, \mu = 16$, then $X$ is the quotient of a
principally polarized  Abelian surface
$A$ by the group $\{\pm 1\}$.

But in char $=2$ \cite{shioda} Kummer surfaces behave differently, and have at most $4$ singular points.

In this paper we show among other results   that, if $X$ is a normal quartic surface defined over an  algebraically closed field $K$ of characteristic $2$, then:
\begin{enumerate}
\item
If $X$ has a point of multiplicity $3$, then $ | Sing(X)| \leq 7$. (Proposition \ref{monoid}).
\item
If $X$ has a point of multiplicity $2$ such that the projection with centre $P$ is inseparable, then $ | Sing(X)| \leq 16$ (Proposition \ref{inseparable}, see steps I) and II)  for smaller upper bounds under special assumptions).
\item
If $X$ has a point of multiplicity $2$ such that the projection with centre $P$ is separable, 
but $P$ is uniplanar (the projective tangent conic at $P$ is a double plane), and moreover  properties i) or ii) or iii) hold, 
then $ | Sing(X)| \leq 19$ (Proposition \ref{separable}).
\item
$ | Sing(X)| \leq 20$ (Theorem \ref{mt}).
\item
If $ | Sing(X)| = 20$,  the minimal resolution
$S$ of $X$ is a minimal K3  surface and the only singularities of $X$ are nodes.
Moreover  $S$ 
has Picard number $\rho(S)$ with $ \rho(S) =   b_2(S) = 22$.
\end{enumerate}

Except for items iv), v), we use elementary methods, but with  these alone we are only able to get  the weaker inequality 
 $ | Sing(X)| \leq 25$ in iv). 

We hope to return in the future to the problem of the existence 
of  birational embeddings as quartic surfaces of K3 surfaces with large Picard number.

\subsection{Notation and Preliminaries} For a point in projective space, we shall freely use  the vector notation $(a_1, \dots, a_{n+1})$,
instead of the more precise  notation $[a_1, \dots, a_{n+1}]$, which denotes the equivalence class of the above vector.

Let $Q(x_1,x_2,x_3)=0$ be a conic over a field $K$ of characteristic $2$.

Then we can write  $$Q(x_1,x_2,x_3)= \sum_i b_i^2 x_i^2 + \sum_{i<j} a_{ij} x_i x_j= ( \sum_i b_i x_i)^2 + \sum_{i<j} a_{ij} x_i x_j.$$
One finds that, unless $Q$ is the square of a linear form,  $[a] : = (a_{23},a_{13},a_{12})$ is the only point where the gradient of $Q$ vanishes. 

Taking coordinates such that $[a] = (0,0,1)$ we have that $Q(x) = x_1 x_2 + b(x)^2$, where $b(x)$ is a
(new)  linear form: for instance, if $Q(x_1,x_2,x_3)= x_1 x_2 + x_1 x_3 + x_2x_3$, then $Q  = (x_1 + x_2)(x_1 + x_3) + x_1^2$.

We have two cases  $Q[a]=0$, hence $b_3=0$, hence $Q(x) = x_1 x_2 + b(x_1, x_2)^2$, hence changing again coordinates
we reach the normal form $Q = x_1 x_2$; while if $b_3 \neq 0$ we reach the normal form  $Q = x_1 x_2 + x_3^2$.

Hence we have just the three normal forms (as in the classical case)
$$ x_1^2 ,  \ x_1 x_2, \  x_1 x_2 + x_3^2.$$
 
\section{Singular points of quartic surfaces in characteristic 2}

We consider a quartic surface $X = \{ F =0\} \subset \PP^3_K$, where $K$ is an algebraically closed field of characteristic equal to $2$,
and such that $X$ is normal, that is, $ Sing (X)$ is a finite set. If $X$ has a point of multiplicity $4$, then this is the only singular point, while 
if $X$ contains a triple point $P$, we can write the equation, assuming that the point $P$ is the point $x_1=x_2=x_3=0$:
$$ F (x_1, x_2,x_3, z) = z G (x) + B (x) ,$$
and, setting $G_i : = \frac{\partial{G}}{\partial{x_i}}, B_i : = \frac{\partial{B}}{\partial{x_i}}$, we have 
$$ Sing(X) = \{ G(x) = B(x) = G_i z + B_i = 0, \ i=1,2,3 \} \ .$$

If $(x,z) \in Sing(X)$ and $x \in  \{ G(x) = B(x) = 0\}$, then $ x \notin Sing (\{G=0\})$, since $ x \in Sing (\{G=0\}) \Rightarrow  x \in Sing (\{B=0\})$ and then the whole line $(\la_0 z, \la_1 x) \subset Sing(X)$. Hence $\nabla (G) (x) \neq 0$ and there exists a unique
singular point of $X$ in the above line. Since the two curves  $ \{ G(x) = 0\}, \{ B(x) = 0\}$
have the same tangent at  $x$ their intersection multiplicity at $x$ is at least $2$, and we conclude:

\begin{prop}\label{monoid}
Let $X$ be  quartic surface $X = \{ F =0\} \subset \PP^3_K$, where $K$ is an algebraically closed field,
and suppose  that $ Sing (X)$ is a finite set. If $X$ has a triple point then $ | Sing(X)| \leq 7$.

More generally, if  $X$ is a degree $d$ surface $X = \{ F =0\} \subset \PP^3_K$, where $K$ is an algebraically closed field,
and we suppose  that 
\begin{itemize}
\item
$ Sing (X)$ is a finite set, and 
\item
  $X$ has a  point of multiplicity $d-1$
  \end{itemize}
  
 then $$ | Sing(X)| \leq 1 + \frac{d (d-1)}{2}.$$

\end{prop}

\begin{proof}
The second assertion follows by observing that in proving the first we never used the degree $d$, except for
concluding that the total intersection number (with multiplicity) of $G,B$ equals $ d (d-1)$.

\end{proof}
 
Assume now that we have a double point $P$ of $X$ and we take coordinates such that $P = \{  x := (x_1, x_2,x_3) =0, z=1\}$, 
thus we can write the equation 
 $$ {\bf (Taylor \ development)}: \  \ F (x_1, x_2,x_3, z) = z^2 Q (x) +  z G (x) + B (x) .$$

Then $$  \ Sing (X) = \{ (x,z) | G(x) =  z^2 Q (X)  + B(x) = z^2 Q_i (x) +  z G_i  (x) + B_i (x) =0, \ i=1,2,3\}.$$

We consider then the projection 
$$\pi_P : X \setminus \{P\} \ra \PP^2 : = \{ (x_1, x_2,x_3)\}.$$

\begin{lemma}\label{1-1}

 1) If $P$ is a singular point of the quartic $X$, the projection $Sing (X) \setminus \{P\} \ra \PP^2$ is an injective map,
except possibly for the points mapping to  the finite subscheme $\Sigma \subset \PP^2$ defined by $Q=G=B=0$,
if the three gradients $\nabla Q (x) ,  \nabla G(x) ,   \nabla B (x)$ are all proportional and  $\nabla Q (x) \neq 0$,
$\nabla G (x) \neq 0$.

2) If $x \in \Sigma$, then there is a $z$ such that $(x,z) \in Sing(X)$ if 
$$ z^2 \nabla Q (x)  + z \nabla G(x)  +  \nabla B (x)= 0.$$
\end{lemma}

\begin{proof}
In fact, if a line $L$ through $P$ intersects $X$  in $2$ other singular points, then $ L \subset X$,
hence 
$$ L \subset \{ (x,z) | Q(x) = G(x) = B(x) = 0\} = : P * \Sigma,$$
where $\Sigma \subset \PP^2$ is the subscheme defined by $Q=G=B=0$,
and $\Sigma$  is a 0-dimensional subscheme since $X$ is irreducible.

If $x \in \Sigma$ and $(x,z), (x,z+ w) \in Sing(X)$ are different points, from the equations 
$$ z^2 \nabla Q (x)  + z \nabla G(x)  +  \nabla B (x)= (z^2 + w^2 )\nabla Q (x)  + (z + w ) \nabla G(x)  +  \nabla B (x)=0$$
follows $   \nabla G(x) = w \nabla Q (x) $, $ \nabla B (x) = z (z+w)  \nabla Q (x)  )$.
In particular, it cannot be $\nabla Q (x) = 0$, and the three curves are all tangent at $x$.

Finally, if the three gradients are proportional, then we can find $(z,w)$ solving the equations
$   \nabla G(x) = w \nabla Q (x) $, $ \nabla B (x) = z (z+w)  \nabla Q (x)  )$; and $w\neq 0$ if $   \nabla G(x) \neq 0$.

\end{proof}

\subsection{The inseparable case}
We consider first the inseparable case where $G \equiv 0$, hence
$$ X  = \{   z^2 Q(x)  + B(x) = 0 \},  Sing (X) =  X \cap \{   z^2 \nabla Q (x) =  \nabla B (x) \}.$$

\begin{prop}\label{inseparable}
Let $X$ be  a normal quartic surface $X = \{ F =0\} \subset \PP^3_K$, where $K$ is an algebraically closed field
of characteristic  $2$,
hence $ Sing (X)$ is a finite set. If $X$ has a double point  $P$ as in (Taylor) 
such that the projection with centre $P$ is an inseparable double cover of $\PP^2$, i.e., $G \equiv 0$,
 then  $ | Sing(X)| \leq 16$, and there exists  a  case with $ | Sing(X)| =  14$.

\end{prop}

\begin{proof}
As in the proof of proposition \ref{monoid}, if a singular point $(x,z)$
satisfies $Q(x)=0$, then $B(x)=0$,
 and since $ Sing(X)$ is finite it must be   $\nabla (Q) (x) \neq 0$:  under this assumption
$(x,z)$ is the only  singular point lying above
the point $x \in \Sigma = \{ Q(x)= B(x)=0\}$.

Hence, in view of Lemma \ref{1-1} the projection $Sing(X) \setminus \{P\} \ra \PP^2$ is injective.

Conversely, if $x \in \Sigma  = \{ Q(x)= B(x)=0\}$,  it is not possible that $\nabla Q (x) =  \nabla B (x)=0$,
while for $\nabla Q (x) = 0,  \nabla B (x) \neq 0$ there is no singular point lying over $x$,
and for $\nabla Q (x)  \neq 0$ there is at most one singular point lying over $x$, and one iff 
$\nabla Q (x) ,  \nabla B (x) $ are proportional vectors. 
  
Hence in the last  case the intersection multiplicity of $Q,B$ at $x$ is at least $2$, and
in particular  over $\Sigma$ lie at  most $4$ singular points.

{\bf Step I) 
We first consider the case where $Q$ is a double line, and show that in this case $X$ has at most $10$ singular points.}

\medskip

We can in fact choose coordinates such that $ F = z^2 x_1^2 + B(x)$, hence the singular points $(x,z)$ are determined
by the equations $F=  \nabla B (x)=0$. And we have a bijection between $Sing(X)$ and the points of the plane with coordinates $x$ where
 $ \nabla B (x)=0$ and   $x_1 \neq 0$. In fact the points where  $ \nabla B (x)=x_1 = 0, B \neq 0$, do not come from singular points,
 while the points with $ \nabla B (x)=x_1 = B = 0$ would provide   infinitely many singular points.

 In general, if $\{ g(x) = 0 \}$ is a plane curve of even degree $d=2k$ , from the Euler formula $\sum_i x_i g_i \equiv 0$,
  we infer that if the critical scheme $\sC_g := \{\nabla g (x) =  0\}$  is finite, then its cardinality is at most $(d-1)^2$ by the theorem of B\'ezout.
  Take in fact a line $L$ not intersecting this scheme $\sC_g$, and assume that $L = \{ x_3 = 0\}$: then 
  $\{ \nabla g (x) =  0\} = \{ g_1 = g_2=0, \ x_3 \neq 0 \}$, and this set has, with multiplicity, cardinality $(d-1)^2$ unless 
   $g_1=  g_2 = x_3 = 0$ is non empty.
   
   We are done if $\sC_B$ is finite. In the contrary case, since $Sing(X)$ is finite, the only common divisor of the $B_i$'s
   is $x_1^a$, and, replacing $B_i$ by $B_i /  x_1^a$, we get a finite subscheme and a better estimate than
   $9$.

   Hence Step I) is proven.
   
   \bigskip

   Remark:   we get smaller upper bounds if $g$ is reducible: for instance, if $g = f \ell$, where $\ell$ is a linear form, then it is easy to see, taking coordinates where $\ell = x_1$,
   that $|\{ \nabla g (x) =  0\} | \leq 1+  (d-2)(d-1) $,
   a bound which for $d=4$ yields $7$ but for $d=6$ yields $21$).
   
\bigskip

 Also for later use, we prove
   
   \begin{lemma}\label{quartic}
   Assume that $B(x_1, x_2,x_3) \in K [x_1, x_2,x_3]$ is a homogeneous polynomial of degree $4$, and that $K$ is an algebraically closed field of characteristic $2$.
   
   Then, if the scheme $\sC_B := \{\nabla B (x) =  0\}$ is finite, its cardinality is at most $7$.
   \end{lemma}
   \begin{proof}   
   Since $x_3 B_3 = x_1 B_1 + x_2 B_2$, we get that if $B_1=B_2 = 0$, then $x_3 B_3=0$.
   
   If $\{B_1=B_2 = 0\}$ is infinite, then on its divisorial part $C$ we have $x_3 =0$, since $\sC_B$ is finite.
   
   This means that $B = x_2 B' + \be(x_1,x_2)^2$, hence $\nabla B=0 \Leftrightarrow \nabla (x_2 B')=0$
   and we can apply the previous remark to infer $ | \sC_B|  \leq 7.$
   
   If $\sB' : = \{B_1=B_2 = 0\}$ is finite, take a line containing $\de \geq 2$ points of $\sB'$, 
   say $x_3=0$, and let $x_1, x_2$ be coordinates 
   both not vanishing at these points.
   
   Then by Bezout's theorem $\sB'$ consists of at most $9$ points, counted with multiplicity.
   
   At any of the $\de \geq 2$ points of  $\sB' \cap \{x_3=0\}$, the equation 
   $x_3 B_3 = x_1 B_1 + x_2 B_2$ shows that either $B_3$ does not vanish, or $x_3 B_3$ vanishes of multiplicity at least $2$,
   hence a linear combination of
   $\nabla B_1, \nabla B_2$ vanishes, implying that  the local  intersection  multiplicity of $B_1, B_2$ is $\geq 2$.
   
   Hence $ | \sC_B| \leq 9 - \de \leq 7.$

   \end{proof}
   
   \bigskip

{\bf Step II) 
We consider next the case where $Q$ consists of two lines, and show that in this case $X$ has at most $13$ singular points.}

\medskip

We can in fact choose coordinates such that $ F = z^2 x_1 x_2 + B(x)$, hence 
$$Sing(X) = \{z^2 x_1 x_2 + B(x)= 0 , z^2  x_2 + B_1(x)=  z^2 x_1  + B_2(x)=  B_3=0\}.$$
Hence the singular points satisfy 
$$ (**) \ B_3= B + x_1 B_1 = B + x_2 B_2 = 0,$$
where the last equation follows form the first two, in view of the Euler relation.

Let $(x)$ be  one of the solutions of $(**)$: we find at most one  singular point lying above it
if $x_1 \neq 0$, or $x_2 \neq 0$.
If instead $x_1 = x_2 =0 $ for this point, then $B =B_3 =0$, and either there is no singular point of $X$ lying over it,
or $(x)$ is a singular point of $B$, and we get infinitely many singular points for $X$, a contradiction.
Hence  it suffices to bound the cardinality of this set.

If the solutions of $(**)$ are a finite set, then their cardinality is $\leq 12$, and also $|Sing(X)| \leq 13$.

If  instead $(**)$ contains an irreducible curve $C$,  this factor $C$ cannot be   $x_1=0$ or $x_2=0$, since 
 for instance in the first case then  $ x_1 | B \Rightarrow x_1 |F$, a contradiction.
 
 We find then infinitely many points satisfying $(**)$ and with $x_1x_2 \neq 0$. Over these lies a singular 
 point if $ z^2 = B_1/ x_2 = B_2 /x_1 =  B / (x_1 x_2)$ is satisfied. 
 
 But if the first equation is verified, then also the others follow from $x_1 B_1 + x_2 B_2 = x_3 B_3 = 0$,
 respectively $ B + x_1 B_1 = 0$.

Hence the existence of such a curve $C$ leads to the existence of infinitely many singular points of $X$
 and Step II) is proven by contradiction.

\bigskip

{\bf Step III)
We consider next the case where $Q$ is smooth, and show that in this case $X$ has at most $16$ singular points.}

\medskip

We can in fact choose coordinates such that $Q(x) =  x_1 x_2 +  x_3^2$,
hence $ F = z^2  Q(x) + B(x) = z^2 x_1 x_2 + z^2 x_3^2 +  B(x)$.

Here
$$Sing(X) = \{z^2 x_1 x_2 +  z^2 x_3^2 + B(x)= 0 , z^2  x_2 + B_1(x)=  z^2 x_1  + B_2(x)=  B_3=0\}.$$
Hence the singular points satisfy 
$$ (*) \ B_3= B + x_1 B_1 + z^2 x_3^2 = B + x_2 B_2 + z^2 x_3^2 = 0,$$
and again here  the last equation follows form the first two, in view of the Euler relation.

If $x_1 \neq 0$, or $x_2 \neq 0$, there is at most one singular point lying over (we mean always:  a point different from $P$) 
the point $x$. The same for $x_1= x_2=0$
since then $x_3^2 \neq 0$.

Multiplying the second equation of (*) with $x_2$, and the last with $x_1$, we get, for the singular points,
$$ B x_2  + Q B_1  =B x_2  + x_1 x_2 B_1 + B_1 x_3^2 = 0$$
$$ B x_1  + Q  B_2 = B x_1  + x_1 x_2 B_2 + B_2 x_3^2 =  0.$$

Hence the singular points different from $P$ project injectively into the set
$$ \sB : = \{x|  B_3 = B x_2  + Q B_1  =B x_1  + Q  B_2= 0 \} \supset  \{x|  B_3 = B   = Q  = 0 \}.$$

By B\'ezout $\sB$ consists of  at most $15$ points, unless  $B_3 , B x_2  + Q B_1,  B x_1  + Q  B_2$ have a
 common component.

If $\sB$  contains an irreducible  curve $C$, since either $x_1 \neq 0$ or $x_2 \neq 0$ for the general point of $C$,
 there exists $ i \in \{1,2\}$ such that the cone $\Ga$ over $C$ has as open nonempty subset the cone $\Ga'$ over $C' : = C \setminus \{x_i=0\}$ which is contained in   the set of  solutions of $(*)$. Hence $\Ga$ is contained in   the set of  solutions of $(*)$.

Argueing similarly,  if  $C \neq \{x_1=0\}, C \neq \{x_2=0\}$, the cone $\Ga$ over $C$ has  as nonempty open subset 
the cone $\Ga''$ over $C'' : = C \setminus \{x_1 x_2 =0\}$ which is contained in the solution set 
of $z^2  x_2 + B_1(x)=  z^2 x_1  + B_2(x)=  B_3=0$, hence $\Ga$ is contained in this solution set. 
 We conclude that  $Sing(X)$ contains $ X \cap \Ga$,
hence it is an infinite set, a contradiction.

By symmetry, it suffices to exclude the possibility that $\sB$ contains $C = \{x_1 = 0\}$.
In this case we would have that  $x_1 | B_3, B_2, B x_2 + Q B_1$. Since $x_1 $ does not divide $Q$, it follows  that the
plane $\Ga  = \{x_1 = 0\}$ has  as nonempty open subset the cone over $C \setminus \{Q=0\}$
which is contained in  the set $\{z^2  x_2 + B_1(x)=  z^2 x_1  + B_2(x)=  B_3=0\}$,  hence
$\Ga$ is contained in this set and  $Sing(X)$ contains $ X \cap \Ga$, again
 a contradiction.

Hence Step III) is proven.

\bigskip

{\bf Step IV) 
We  construct now a case where there are other $13$ singular points beyond $P$, hence $X$ has $14$ singular points.}

By the previous steps, we may assume that $Q(x) =  x_1 x_2 +  x_3^2$, and we take $B = y_1 y_2 y_3 y_4$,
where  $y_1,  y_2,  y_3$ are independent linear forms and  $y_4 =  y_1 + y_2 + y_3$.

Recall that 
$$ Sing (X)  = \{   z^2 Q(x)  + B(x) = 0 ,   z^2 \nabla Q (x) =  \nabla B (x) \}.$$
Multiplying the second (vector) equation by $Q$ we get 
the equation 
  $$ (*) \  B(x)   \nabla Q (x) =  Q(x)  \nabla B (x) \Leftrightarrow   \nabla (QB) (x) =0.$$

  The solutions of (*) consist of 
  \begin{enumerate}
  \item
  the points where $Q(x) = 0 $, hence $Q=B=0$: these are precisely $8$ points, for general choice of the linear forms $y_i$;
  and they are not projections of singular points of $X \setminus{P}$, as the gradients $ \nabla B,  \nabla Q$ are linearly independent
  at them;
    \item
  the points where $Q(x) \neq 0 $, $B(x)=0$,  hence $ B = \nabla B =0$; 
   that is, the $6$ singular points $y_i = y_j=0$ of $\{ B(x)=0\}$, 
  giving rise to the $6$ singular  points of $X$ with $z=0$,
  \item
  (possibly) a point where $\nabla Q (x) =    \nabla B (x)=0$ but $Q \neq 0, B\neq0$, this comes from exactly one singular point of $X$;
    \item
  points satisfying $$(**) \ Q(x) \neq 0 \neq  B(x), \nabla (Q) \neq 0 \neq \nabla (B), B(x)   \nabla Q (x) =  Q(x)  \nabla B (x).$$.
  \end{enumerate}
  
Observe that  $\nabla Q (x) = (x_2,x_1,0) $ vanishes exactly at the point $x_1=x_2=0$, while 
in the coordinates $(y_1, y_2, y_3)$ we have
$$\ ^t \ \nabla (B) (y) = (y_2  y_3 (y_2 + y_3), y_1  y_3(y_1 + y_3),y_1 y_2 (y_1 + y_2)), $$
hence the gradient $\nabla (B) (y)$ vanishes exactly at the $6$ singular points of $B$, and at the point $y_1+ y_2=y_1 +y_3=0$.

We have that this point is the point $x_1=x_2=0$ as soon as $y_1 = y_3 + x_1, y_2 = y_3 + x_2$ (then $y_4 =  y_3 + x_1 + x_2$).

Going back to the notation of Step III), we consider then the set  
$$ \sB : = \{x|  B_3 = B x_2  + Q B_1  =B x_1  + Q  B_2= 0 \} \supset  \{x|  B_3 = B   = Q  = 0 \}.$$

Since $ B = (y_3 + x_1)(y_3 + x_2) y_3 (y_3 + x_1+ x_2)$, if we set $ y_3 = a_1 x_1 + a_2 x_2 + a_3 x_3$, with $a_3 \neq 0$, 
then using
$$B_i = (y_1 y_2 y_3 y_4)_i = \sum_1^4  \frac{B}{y_j} (y_j)_i, \ $$
 $$ B_3 = 0 \Leftrightarrow   \sum_1^4  \frac{B}{y_j} = 0, \ B_1 = a_1 B_3 + y_2 y_3 y_4 + y_1 y_2 y_3, \ B_2 = a_2 B_3 + y_1 y_3 y_4 + y_1 y_2 y_3,$$
 the equations of $ \sB$ simplify to
 $$ B_3 = 0 , \ y_2 y_3 ( y_1 y_4 x_2 + Q (y_1 + y_4)) = 0 , y_1 y_3 (y_2 y_4 x_1 + Q (y_2 + y_4) = 0.$$ 
 
 Since we are left with finding solutions where $B\neq0$, the equations reduce to
 $$ B_3 = 0 , \  [(y_3 + x_1) (y_3 + x_1+ x_2)   + Q ] x_2  = 0 ,  [(y_3 + x_2) (y_3 + x_1+ x_2)  + Q ] x_1 = 0.$$
 We already counted the point $x_1 = x_2 =0$. If $x_1 =0$ and $x_2 \neq 0$, 
 we find $$ B_3 = y_3^2 + y_3 x_2 + x_3^2 = 0 \ ,$$
 but since we observe that $B_3 \equiv 0$ on the line $x_1=0$, we get the two points 
 $$ x_1 = y_3^2 + y_3 x_2 + x_3^2 = 0 \ \Leftrightarrow x_1=0, (a_3^2 +1) x_3^2 + (a_2^2 + a_2) x_2^2 + a_3 x_2 x_3=0.$$
 Similarly 
  if  $x_2 =0$ and $x_1 \neq 0$
 we get  the two points $$ x_2 = y_3^2 + y_3 x_1 + x_3^2 = 0 \  \Leftrightarrow x_2=0, (a_3^2 +1) x_3^2 + (a_1^2 + a_1) x_1^2 + a_3 x_1 x_3=0.$$
 If both $x_1 \neq 0, x_2 \neq 0$, 
 we find $$ B_3 =x_1^2 + x_2^2 + y_3 (x_1 + x_2) = (y_3 + x_1)^2 + y_3   x_2   + x_3^2= 0 \ .$$
 The second equation (of the three above)  is reducible, it equals $$(x_1 + x_2 + y_3) (x_1 + x_2)=0.$$
  Again $B_3 \equiv 0$ on the line $x_1=x_2$,
 while the points with $(x_1 + x_2 + y_3)=0$ yield the line $y_4=0$ which is contained in $B$, hence we do not need to consider these points.
 
 Hence we get  two more solutions:
 $$ x_1 = x_2 , y_3^2 + x_1^2 + y_3   x_1   + x_3^2= 0,\Leftrightarrow$$
 $$ \Leftrightarrow x = (x_1, x_1, x_3) , (a_3^2 +1) x_3^2 + (a_2^2 + a_1^2 + a_2 + a_1+ 1) x_1^2 + a_3 x_1 x_3=0.$$
 and $X$ has exactly $1 + 6 + 1 + 2+2+2= 14$ singular points,
 $$ (1): x=0, z=1, $$
 $$ (6):  z=0, x = (0,1,a_2), (0,1, 1 + a_2), (1,0, a_1), (1,0, 1 + a_1), (1,1, a_1 + a_2), (1,1,  1 + a_1 + a_2), $$
 $$ (1): ( 1,0,0,1),$$
 $$ (2):  x= (0, 1, b) , \ (a_3^2 +1) b^2 + (a_2^2 + a_2)  + a_3 b =0$$
 $$ (2):  x= (1,0, c) , \ (a_3^2 +1) c ^2 + (a_1^2 + a_1)  + a_3 c =0.$$
 $$ (2):  x= (1,1, d) , \  (a_3^2 +1) d^2 + (a_2^2 + a_1^2 + a_2 + a_1+ 1)  + a_3 d =0.$$
 One can now verify that, for general choice of the $a_i$'s (which can be made explicit requiring 
 $ b \neq a_2, 1 + a_2$, $ c \neq a_1, 1 + a_1$, $ d \neq a_1 + a_2, 1 + a_1 + a_2$), we obtain $14$ distinct points.
 
\end{proof}
 
 \bigskip
 \subsection{The separable case}
\begin{prop}\label{separable}
Let $X$ be  a normal quartic surface $X = \{ F =0\} \subset \PP^3_K$, where $K$ is an algebraically closed field
of characteristic  $2$
(so  $ Sing (X)$ is a finite set). 

I) If $X$ has a double point  $P$  
such that the projection $\pi_P$ with centre $P$ is a separable double cover of $\PP^2$, then  $ | Sing(X)| \leq 25$.

II) If moreover $P$ is a uniplanar double point (i.e., the tangent quadric at $P$ is a double plane  $Q = x_1^2$),  and 

i) $x_1$ divides $G$,  then $ | Sing(X)| \leq 17$;
while  

ii) if  $  G(x) = x_1^2 L(x_1, x_2) + x_2 x_3 N (x_1 , x_2),$ then $ | Sing(X)| \leq 19$,

  (iii) if $$ \  G(x) = x_1^2 L(x_1, x_2,x_3) +  x_1 x_2  M(x_2, x_3) +  x_2^3,$$
then also $ | Sing(X)| \leq 19$.
\end{prop}
 \begin{proof}
 
 I)
 
We know that $Sing (X)$ is the set
$$   \{ G(x) =  z^2 Q (x)  + B(x) = z^2 Q_i (x) +  z G_i  (x) + B_i (x) =0, \ 1 \leq i \leq 3\},$$
and that the cubic polynomial $G(x) $ is not identically zero. 

Except for the equation $G(x)=0$, the other equations involve $z$ and we look for the set $M$ of points $x$
such that there exists a common solution $z$ for these $4$ equations. Because of Lemma \ref{1-1}
we know that the set $Sing(X) \setminus \{P\} \setminus \{ Q(x)=0\}$ injects into $M$. 

\smallskip

The strategy shall be to  find some equations defining  sets $M' \subset \sB$ larger than $M$,
and try to bound the cardinality of $\sB$. In the end, if we cannot show that $\sB$ is finite, we shall show that 
$M'$ is exactly the image of  $Sing(X) \setminus \{P\}$, thereby showing that $M' = M$, hence it is finite,
and then bounding the  cardinality of $M'$ from above.

Multiplying the last three equations by $Q$ we obtain:
$$ B Q_i + z Q G_i + Q B_i =0, i=1,2,3.$$

If now $Q$ is not a square, we may assume that $Q=x_1 x_2$ or $Q=x_1 x_2 + x_3^2$, hence for $i=3$
 we get $$Q ( z G_3 + B_3)=0.$$

and eliminating $z$  we get $ Q f_{ij} = 0, \ i,j = 1,2,3$ where 
 $$  f_{12} : = (B Q) _1   G_2 -  (B Q) _2   G_1 = Q (B_1 G_2 + B_2 G_1 ) + B ( Q_1 G_2 + Q_2 G_1) =$$
$$ =  Q (B_1 G_2 + B_2 G_1 ) +  B  x_3 G_3,$$
because of Euler's formula, and 
 $$   f_{i3} : = Q G_i B_3 + G_3 (B Q_i +Q B_i) = Q ( G_i B_3 + G_3B_i) + G_3 B Q_i \ i=1,2.$$
 
 Hence either $Q=0$ or $ f_{ij}  =0, \forall  i,j =1,2,3.$
 
 Observe that the singular points where $Q(x)=0$ are the points of the   finite subscheme  $\Sigma$ defined by 
  $Q(x) = G(x) =B(x) = 0$, such that 
 the cone $ P * \Sigma$ is contained in $X$. 
 
By Lemma \ref{1-1}  for  the points of $Sing(X)$ lying over $\Sigma$,  $z$ is uniquely determined unless the gradients  
$\nabla G, \nabla Q, \nabla B$ are all proportional, with $\nabla Q, \nabla G \neq 0$.

Moreover
 the set $\Sigma$ has cardinality at most $8$, equality holding only if $ Q | G$ and $Q$ is not a square.
  But in this case 
 over a point of $\Sigma$ lies a point of $Sing(X)$ only if the gradients
 $\nabla B, \nabla Q$ are linearly dependent, hence  the number of such singular points is at most $8$,
 and at most $6$ if $ Q$ does not divide $G$, as it follows by the theorem of B\'ezout also in the case where $Q,G$ have a common factor of degree $1$.
 
 The singular points of $X$ for which $Q(x) \neq 0$  project in a $1-1$-way  to the points of the subscheme  
 $$\sB : = \{ x| G(x) = f_{ij}(x) = 0, i,j =1,2,3\}.$$
 If the subscheme $\sB$ is   finite, it is contained in a  finite set of the form 
 $$\{ G(x) =  \sum_{ij} \la_{ij}f_{ij}(x) = 0 \},$$ for general choice of the $\la_{ij}$'s,
 hence its  length is at most $ 3 \cdot 7 = 21$ by the B\'ezout theorem.
 
 It is immediate to see that in this case  the scheme   $ \Sigma = \{ Q(x) = G(x) =B(x) = 0\}$
 is a  subscheme of $ \sB $, so we conclude  that   the number of singular points of $X$ is at most $ 1 +  21= 22$.
 
 If the subscheme $\sB$ is not finite, we observe that singular points fulfill also the degree ten equation
 obtained by eliminating $z$,
 $$ B Q G_i^2 = (BQ_i + QB_i)^2.$$
 Since we assumed that $Q$ is not a square, hence $Q=x_1 x_2$ or $Q=x_1 x_2 + x_3^2$,  for $i=3$
 we get, dividing by $Q$, the degree $8$ equation
  $$ B  G_3^2 =  QB_3^2.$$
Also the equations $f_{ij}$, for $j=3$, are, as just seen,  $(B Q) _i   G_3 -  Q B _3   G_i = 0 $.

If $G'$ is a common component of $G, f_{ij}$, and has degree $1$, then we get at most $ 1 \cdot  8 + 2 \cdot 6=20$
points, if $G'$ has degree $2$, we get at most $2 \cdot  8 + 1 \cdot 5=21$ points, and $X$ has at most $22$
singular points.

Finally, if  $G$ divides $f_{ij}$ for all $i,j$,  we get the estimate $ \leq 1 + 24 = 25$;
because if we had a curve $C$ contained in $G=0$ where $ B  G_3^2 +  QB_3^2=0$,
over the points of $C$ we find $z$ such that $ z G_3 + B_3=0, z^2 Q + B=0$; in view of the vanishing
of all $f_{ij}$'s we would obtain infinitely many singular points.

{\bf End of the proof of I) plus proof of  II).}

If $Q = x_1^2$, then $Sing (X)$ is the set
$$   \{ G(x) =  z^2 x_1^2  + B(x) =  z G_i  (x) + B_i (x) =0, \ 1 \leq i \leq 3\},$$
and  the last three equations imply 
$$\phi_{ij} : = G_i B_j - G_j B_i =   0 , i,j =1,2,3.$$

0) Assume now that $x_1 $ does not divide $G$.

This time the scheme $\Sigma$ consists of at most $3$ points, 
and over the points of $\Sigma$ lies at most one point of $ Sing(X)$: this is clear if $\nabla G \neq 0$
at them, otherwise there is none since $\nabla B \neq 0$ because of the fact that $ Sing(X)$ is finite.

As in case I) we conclude that if the scheme $$\hat{\sB }: = \{ x| G(x) = \phi_{ij}(x) = 0, i,j =1,2,3\}$$
is finite, it has length at most $3 \cdot 5 = 15$, hence $ | Sing(X)| \leq 16$.

If instead $\hat{\sB }$ is infinite, there is a GCD $G'$ of $G(x) ,  \phi_{ij}(x), {\rm for}\  i,j =1,2,3 $,
of degree $\geq 1$. 

This GCD $G'$ cannot be divisible by $x_1$, hence we claim that the set $M' : = \hat{\sB } \cap \{ B G_i^2 - B_i^2 x_1^2=0, \ i=1,2,3\}$
is finite.  In fact, for the points in the curve $G' =0$ , and with $x_1 \neq 0$, there exists a solution of the $4$ equations involving $z$, and since we cannot  have infinitely many singular points, $M' \cap \{ x_1 \neq 0\}$ is finite, and on the other hand $G = x_1 =0$
is finite. Hence $M'$  is finite.

Again, if  $G'$   has degree $1$, we get,
in view of the degree $8$ equations $ B G_i^2 = B_i^2 x_1^2$,
at most $ 1 \cdot  8 + 2 \cdot 4=16$ points, and if $G'$ has degree $2$, we get 
at most $ 2 \cdot  8 + 1 \cdot 3=19$ points, hence  $ | Sing(X)| \leq 20$.

In the case  that $G$ divides $\phi_{ij}$ for all $i,j$, we get instead   the estimate
 $ \leq 1 + 24 = 25$. 
  
 \bigskip

i) Assume  now that $x_1$ divides $G$, so that $ G = x_1 q$.

For the singular points with $x_1=0$, we get $B=0$, hence at most $4$ points since $x_1$ does not divide $B$.

For the singular points with $q=0$, and $x_1 \neq 0$, since $G_1 = q + x_1 q_1, G_2 =  x_1 q_2, G_3 =  x_1 q_3$,
we observe that 
$$ \{q = \phi_{ij} = 0\}  =  \{q= q_i B_j - q_j B_i = 0\},$$
hence we get at most $8$ points unless $q, q_i B_j - q_j B_i $ have a common component $g'=0$.

But in this case we can use, instead of the above degree $8$ equations, the degree $6$ equations 
$ B q_i^2 - B_i^2 = 0$, and get by the same argument as before
at most $12$ points (since we do not have infinitely many singular points).

Hence we get at most $ 1 + 4 + 12 = 17$ singular points in this case.

\bigskip

Assume now that
$$ ii) \  G(x) = x_1^2 L(x_1, x_2) + x_2 x_3 N (x_1 , x_2),$$
where $x_1$ does not divide $N$. In particular $G$ is irreducible unless $x_2 | L$, 
or $L, N$ are proportional; but the latter case reduces to the first by a change of variables ($y_2 : = N$).

Hence we may assume that either $G$ is irreducible or   $G$ is reducible   with  normal form 

$$ (ii-b) \  G(x) =   x_2 (L_2 x_1^2 +   x_3 N (x_1 , x_2)),$$
which amounts to $L_1=0$ (we have moreover $N_2 \neq0$).

We have in general
$$G_1 =  L_1 x_1^2 + N_1  x_2 x_3, G_2 = L_2  x_1^2 +  N_1 x_1 x_3, G_3 = x_2 N. $$
Then
$$ f_{12} = x_1^2 (L_1 B_2 + L_2 B_1) + x_3 N_1 ( x_2 B_2 + x_1 B_1), \  f_{13} = B_1 x_2 N +   B_3 (  L_1x_1^2+N_1  x_2 x_3),$$
$$ \  f_{23} = B_2 x_2 N + B_3 ( L_2  x_1^2 +  N_1 x_1 x_3 ).$$

 $\sB : = \{ x| G(x) = f_{ij}(x) = 0\}$ consists then of at most $15$ points (and we have then $ |Sing(X) | \leq 16$) 
 unless all the $f_{ij}$ have a common factor with $G$, so they are 
 divisible by $G$ if $G$ is irreducible. 
 
 Assume therefore that $G$ is irreducible and assume that $G$ divides all the $f_{ij}$'s.
 
 Consider now the equations $z^2 x_1^2 = B , z G_i = B_i$, which for $i=2$
 yields $ z x_1 (L_2x_1 + N_1 x_3) = B_2$. Then we obtain 
 $$f: =  B (L_2x_1 + N_1 x_3)^2 + B_2^2 = 0,$$
 an equation of degree $6$ which on $\{G=0\}$ determines at most $18$ points unless it vanishes identically
 on $\{G=0\}$. If $L_2 \neq 0,$ or $N_1 \neq 0$, then at the general point of $\{G=0\}$ $G_2$ does not vanish,
 and setting $ z = \frac{B_2}{G_2}$ equations $f=0, f_{ij}=0$ ensure that we get that over each general point of $\{G=0\}$
 lies a singular point of $X$, a contradiction.
 
 If $L_2 = N_1=0$ the equation of $G$ reduces to $L_1 x_1^3 + N_2 x_2^2 x_3$,
 which is irreducible if and only if $L_1 N_2 \neq0$, and in this case we may assume 
 $ G = x_1^3 +  x_2^2 x_3$. Hence $G_2=0, G_1 = x_1^2, G_3 = x_2^2$.
 
 If $G$ divides the $f_{ij}$, this means that $B_2, x_1^2 B_3 + x_2^2 B_1$ vanish on $\{G=0\}$.
 
 Because of the equation $ z x_1^2 = B_1$ we see that for the singular points the equation
 $$ B_1^2 + x_1^2 B=0$$
 is satisfied.
 
 This equations vanishes over at most $18$ points of $\{G=0\}$, and in this case we are done,
 or it vanishes identically: but then over the general point of $\{G=0\}$ we set $ z : = \frac{B_1}{x_1^2}$
 and we have found a singular point of $X$, because then all other equations $B = x_1^2 z^2 , B_i = z G_i$
 are satisfied: hence a contradiction.

 \begin{rem}
 Here is another  argument for the case where $G$ is irreducible:
 
If  $G$ is irreducible  
 we can eliminate  $x_3$ using $G=0$, and obtain that if $G$ divides these equations, then $G$ divides also

  $$  x_2 N (L_1 B_2 + L_2 B_1) + L N_1 ( x_2 B_2 + x_1 B_1),$$
 $$ B_1 x_2 N^2 + x_1 ^2 B_3 (L_1  N + N_1  L),$$
  $$ B_2 x_2^2  N^2 + x_1 ^2 B_3 (L_2  Nx_2 + N_1 x_1  L).$$

 Consider now the equations $B G_i^2 = B_i^2 x_1^2$. For $i=3$, we get 
 
 $$ B x_2^2 N^2 = B_3^2 x_1^2 \Rightarrow x_2 N^2 ( B x_2 (L_1  N + N_1  L) + B_3 B_1 ) = 0. $$
 Since $G$ is irreducible, then  $\{G =  ( B x_2 (L_1  N + N_1  L) + B_3 B_1 ) = 0\}$ is a finite set $\sM$,
 with at most $18$ points (otherwise, over  the infinitely many points with $ G=0, x_2 N \neq 0$ there lies a singular point with 
 $ z = \frac{B_2}{x_2 N}$). 
 
 This finite set  $\sM$ contains the points images of the singular points, since if 
 we have a singular point with $G_3 = x_2 N = 0$, then $B_3=0$, and 
 
 1) the points with $x_2 = G = B_3=0$ are in $\sM$
 
 2) the points with  $N=G = B_3= 0$ are also in $\sM$, because   then  $x_1L=0$; if $L=0$ we get a
 point of $\sM$, if instead  $N= x_1 = 0$,
  since $x_1$ does not divide $N$, $N= x_1 = 0 \Rightarrow x_1 = x_2 = 0 \Rightarrow L =0$ and we are done.
\end{rem}
 \bigskip
 
 {\bf The case (ii-b)}
$$ (ii-b) \  G(x) =   x_2 (L_2 x_1^2 +   x_3 N (x_1 , x_2)),$$
we note that $G_2$ vanishes only over a finite set of points of $\{G=0\}$, unless $L_2=0$.  If $L_2=0$,
$G$ is the union of three lines 

$$ (ii-c) \ G = x_2 x_3 ( x_2 + N_1 x_1).$$

In the former case where $G_2$ vanishes only over a finite set of points of $\{G=0\}$, 
we have $ z = \frac{B_2}{G_2}$ at outside a finite number of points.

If   equation
 $$f: =  B (L_2x_1 + N_1 x_3)^2 + B_2^2 = 0,$$
 which is verified for the singular points,  
does not vanish on any component of $\{G=0\}$, we get at most $18$ points.

If $q = GCD (G, f_{ij})$, then over $\{q=0\}$ we get at most $6 \cdot deg(q)$ points, 
since if $f$ would vanish identically we would have infinitely many singular points for $X$.

In the whole, we get $$ |Sing(X)| \leq 1 + 6 \cdot deg (q)  + 5 \cdot ( 3 -  deg (q)) \leq 19.$$

In the latter case of three lines, 
$$ (ii-c) \ G = x_2 x_3 ( x_2 + N_1 x_1) \Rightarrow  G_1 = N_1 x_2 x_3, \ G_2 = 0, \ G_3 = x_2  ( x_2 + N_1 x_1)=:  x_2 N.$$

Hence $\nabla G $ vanishes identically on $x_2=0$, therefore over this line we get at most $3$ singular points
unless $B_1, B_2, B_3$ are divisible by $x_2$. In this case setting $ z^2 := \frac{B}{x_1^2}$ we get the contradiction
that $X$ has infinitely many singular points.

Over the line $x_3 =0$ $G_1, G_2$ vanish and  $G_3$ does not vanish identically,
similarly on the line $N=0$ $G_2, G_3$ vanish and $G_1$ does not vanish identically, unless this line is again the line 
$x_2=0$, which we already took into consideration.

Over the line $x_3 =0$   we have at most $3$ points unless $B_1, B_2$ vanish identically.

Setting $ z = \frac{B_3}{G_3}$, we see that we have infinitely many singular points if the equation 
$h: = x_1^2 B_3^2 = B G_3^2$ vanishes identically. Otherwise, we would get at most $8$ points with multiplicity,
an  inequality which is unsatisfactory; however, 
writing  $B = x_3 B' + b ( x_1, x_2)$,  the condition that $x_3 | B_1, B_2$ implies that 
$ b ( x_1, x_2)$ is a square, $ b ( x_1, x_2) = \be(x_1, x_2)^2$.

Then for the singular points lying above the line $\{ x_3=0\}$ the equation $x_1^2 z^2 = B $
boils down to  $x_1 z = \be $, and together with the other equation $ z G_3 = B_3$
implies the equation $ h' = x_1 B_3 + \be G_3 =0$. Now we have replaced the  equation 
$h$ of degree $8$ by an equation $h'$ of degree $4$, hence we have at most $4$ points over the line 
$\{ x_3=0\}$.

Similarly for the third line. 

\medskip

Summing up, $ | Sing(X)| \leq 1 + 3 + 4 + 4  = 12$.

\medskip

 \begin{rem} 
 Here is another argument for the case (ii) where $G$ is reducible.
 
 In the case (ii-b) we have $ L_2 = 1$, $L_1 = 0$, and we are done as before if $x_2$ does not divide the
 $ GCD (G, f_{ij})$: since we get either $1 + 15$ singular points, or $1 + 5 + 2\cdot 6 = 18$ points.
 
 Remains  the possibility that  $x_2  | GCD (G, f_{ij})$.
 This implies that $x_2 | B_3$ (since $x_2 | f_{23}$), and then $x_2 | B_1$ (since $x_2 | f_{12} = x_1^2 B_1 + x_3^2 N_1 B_3$).

  We get at most $10$ points lying above $G/x_2 = 0$ if   
 $f_{ij} = G/x_2 = 0, \forall i,j$ is a finite set, else we get at most $12$ points from the equations $ G/x_2 =  ( B x_2 (L_1  N + N_1  L) + B_3 B_1 ) =0$,
 since  $ G/x_2 = x_1^2 + x_3 N$ is relatively prime with $x_2 N$.

 Write $B = x_2 B' + \be( x_1, x_3)$: since $x_2$  divides  $B_1, B_3$ we infer that $\be$ is a square $\be = q (x_1, x_3)^2$.

 Hence $B_2 = B' + x_2 B'_2$.

  Since  $x_2$  divides  $B_1, B_3, G_1, G_3$,
 from the equation 
 $ B G_2^2 + x_1^2 B_2^2 = x_2=0$ we get 
 $$x_2 =  x_1^2 (B'^2 + (x_1+ N_1 x_3)^2  q^2 )= 0 \Leftrightarrow x_2 =  x_1 (B' + (x_1+ N_1 x_3) q )= 0.$$

 These are at most $4$ points, unless  $(B' + (x_1+ N_1 x_3) q )$ is also  divisible by $x_2$,
 this means that $ B = x_2   (x_1+ N_1 x_3) q + q^2 + x_2^2 h(x)$.

 \bigskip
 
 However in this case $F = x_1^2 z^2 +  x_2 (z G' (x) + B'(x)) + q^2$
 shows that all partial derivatives vanish automatically  for $x_2=0$, 
 except $\frac{\partial F}{\partial x_2} = zG' + B'  + x_2 B'_2$.
 
 Hence in the plane $x_2=0$ the singular set is defined by
 $$zG' + q (x_1+ N_1 x_3) = x_1^2 z^2 +   q^2 =0 \Leftrightarrow zG' + q (x_1+ N_1 x_3) = x_1 z +   q =0.$$
 Then we get $ z (G' + x_1 (x_1+ N_1 x_3))= x_1 z +   q =0$. 
 
 However $(G' + x_1 (x_1+ N_1 x_3)) = x_3 N_2 x_2$,
 hence the singular set contains $\{ x_2 = x_1 z + q=0\}$ which is an infinite set, a contradiction.

 Hence we have a total of at most $ 1 + 17  =18$ singular points in case (ii-b).

 We conclude that in case ii) there are at most $19$ singular points.
  \end{rem}

\bigskip

{\bf  Case (iii).}
  
 In this case
 $$ (iii)  \  G(x) = x_1^2 L(x_1, x_2,x_3) +  x_1 x_2  M(x_2, x_3) +  x_2^3$$
 we have $G_3 =  x_1 ( L_3 x_1 + M_3 x_2),$ $ G_1 = L_1 x_1^2  + M_2 x_2 ^2, \ G_2 = L_2 x_1^2 + M_3 x_1 x_3 + x_2 ^2.$
  
  If $x_1 = 0$ for a singular point,    then $G= 0 \Rightarrow x_2=0$, and $\nabla G$ vanishes at the point $x_1=x_2=0$,
  hence either there is no solution $z$, or each $z$ is a solution, a contradiction. Hence we have shown that for 
  the singular points $x_1 \neq0$. 
  
  As before,   if the scheme $$\hat{\sB }: = \{ x| G(x) = \phi_{ij}(x) = 0, i,j =1,2,3\}$$
is finite, it has length at most $3 \cdot 5 = 15$, hence $ | Sing(X)| \leq 16$.

Assume therefore that this scheme is infinite, and contains a subdivisor  $\{q=0\} $ of $\{G=0\} $.

Then from the equation $ z G_3 = B_3$ we get 
$$ z^2  x_1^2 ( L_3 x_1 + M_3 x_2)^2 = B_3^2 \Rightarrow  B ( L_3 x_1 + M_3 x_2)^2 + B_3^2=0.$$
This is a degree $6$ equation, and if it intersects $\{q=0\} $ in a finite set, this set has at most $ deg(q) \cdot 6$ points, 
 we have in this case 
 $$|Sing(X) | \leq 1 + deg(q) \cdot 6 + (3 - deg(q))  \cdot 5 \leq 19.$$ 
 
 Again, if this set contains a curve $C$, over the general points of the curve $C$, provided $G_3 \neq 0$, 
 setting $ z := \frac{B_3}{G_3}$ the equation $ x_1^2 z^2 = B$ is also satisfied, and 
 yields a value of $z$ corresponding to a singular point,  because  all the $\phi_{ij}$ vanish on $C$.
 
Since we already saw that $x_1 \neq 0$, we are done in the case:

a) $M_3 =0, L_3 \neq 0$.

There remain the possibilities that 

b) $M_3 \neq 0$, $C$ equals $ \Lambda : = \{( L_3 x_1 + M_3 x_2)=0\}$,
or that 

c) $L_3 = M_3 = 0$.

{\bf Case  b):} here  we can change variables and set first  $M : = y_3$, then we reach the  form 
$$ (iii-b)  \  G(x) = x_1^2 L(x_1,  x_3 ) + x_1 x_2 x_3 + x_2 ^3,$$
 by taking as new $x_3$ variable $ (y_3 + L_2 x_1)$; that is,  we can assume  $L_2=0$, $M= x_3$.

In case b)  $G_1, G_2$ are relatively prime; hence, since   $G_3$
vanishes on  $ \Lambda$, so must also by our assumption $ G_1 B_3, G_2 B_3$, hence $B_3$ vanishes on  $ \Lambda$.

Now $G_1 = L_1 x_1^2$, and if $L_1 \neq 0$, then on the points of $ \Lambda$ setting $ z = \frac{B_1}{L_1 x_1^2}$
we find a singular point (since $B_3, G_3, \Phi_{12}$ vanish), once the following equation
$ L_1^2 B x_1^2 + B_1^2 =0$ is satisfied. Hence this equation is satisfied in at most $6$ points of $ \Lambda$,
and we get again the estimate $ |Sing(X)| \leq 19$.

If instead $L_1 = 0$, we find a contradiction, since $\ell : = L_3 x_1 + x_2$ 
does not divide $G = \ell x_1 x_3  + x_2^3$ unless $L_3=0$.

Finally, if $L_1 = L_3=0$, then $G = x_2 (x_1 x_3 + x_2^2)$, and $\Lambda = \{ x_2 =0\}$,
and it suffices to show that over this line we have at most $6$ points.

Over a point of   $\Lambda = \{ x_2 =0\}$ $G_1 , G_3$ vanish hence we find a singular point only if
$B_1, B_3$ vanish: these are at most $3$ points unless $ x_2 | B_1, B_3$. But if this happens,
setting $ z = \frac{B_2}{G_2} = \frac{B_2}{x_1 x_3}$, we find a singular point only if the 
degree $6$ equation $ B x_3^2 + B_2^2=0$ is satisfied. This equation is then manifestly satisfied
in at most $6$ points; indeed, it is satisfied in   at most $3$ points, since  $ x_2 | B_1, B_3$ implies that
$ B(x) = x_3 B' + \be (x_1, x_3)^2$.

{\bf Case  c):}  $L_3= M_3 = 0$, and $G$ is a product of $3$ linear forms in the variables $x_1, x_2$:
 therefore changing variables we can write 

$$ (iii-c)  \  G(x) = x_2 (x_2 + a x_1 ) (x_2 + b x_1).$$

We shall now 
show that we get the upper bound  $19$ also in this case.

Indeed in this case 
$$G_3 = 0, \ G_1 = (a+b) x_2^2, \ G_2 = x_2^2 + ab x_1^2.$$

If $G$ consists of three distinct  lines, it suffices to give an upper bound for the number of singular points lying above one of the three lines, 
without loss of generality it suffices to consider the line $ \Lambda = \{x_2=0\}$.

Since $ G_1 = G_3 =0$ on $ \Lambda$, we have singular points only over the points where $B_1= B_3=0$.

These are at most $3$ points unless $ x_2 $ divides $B_1, B_3$. 

In the latter case $B = x_2 B' + \be(x_1, x_3)^2$ and  on the points of $ \Lambda$  we find 
a singular point  setting $ z = \frac{B_2}{G_2} =  \frac{B_2}{ab x_1^2}$ (observe that   $ab \neq 0$) 
if the  equations 
$$ x_2= (ab)^2 x_1^2 B + B_2^2 =0 \Leftrightarrow x_2= x_1 \be + B_2 = 0 $$ 
are  verified. Hence the solutions are at most $3$ points.

If instead we have a line  $ \Lambda$ of multiplicity at least $2$, then, on the points of $ \Lambda$, we have the vanishing
 $\nabla G =0$, 
hence we have a singular point only over the points where 
$B_1, B_2, B_3$ vanish: these are at most $3$ points unless $B_1, B_2, B_3$ vanish identically.

In this last case, setting $ z^2 = \frac{B} {x_1^2}$, we find that there are infinitely many singular points lying over $ \Lambda$,
a contradiction.

 \end{proof}

 \subsection{The case where the minimal resolution of $X$ is a K3 surface (the singularities are Rational Double Points)}
 The  general estimate $\mu(4) \leq 25$ of (I) of proposition \ref{separable} is by far   not  satisfactory.

 Our strategy  to improve this estimate (developed  after  a conversation with Stephen Coughlan)
  is based on the following remark, 
 which   shows that we  get a better inequality  in the case where 
  we can show that  the minimal resolution $S$ of the  quartic $X$ with double points is  a K3 surface.
  
  Our  strategy is to first give some characterization of the cases where  the minimal resolution $S$ of the  
  quartic $X$ with double points is  not a K3 surface, and then use the previous estimates considered in 
  cases i),  ii) and iii) of Proposition \ref{separable}.
 
 \begin{rem}\label{K3}
 Let $f : S \ra X$ be the minimal resolution of a quartic surface with only double points $p_1, \dots, p_k$ as singularities.
 
 Then the inverse image of each $p_i$ is a union of irreducible curves $E_{i,1}, ,\dots, E_{i, r_i}$
 and if $ r : = \sum_1^k r_j \geq k$, we have irreducible curves $E_1, \dots, E_r$ such that the intersection matrix 
 $\langle E_i , E_j\rangle$ is negative definite \cite{mumford}. 
 
 Hence it follows that the Picard number $\rho(S) \geq r+1 \geq k+1$, keeping in consideration that the hyperplane section $H$ is orthogonal to the $E_i$'s.
 
 Then, using $\ell$-adic cohomology, we obtain a lower bound for the second Betti number $ b_2(S) \geq \rho(S) \geq k+1$, see for instance
 \cite{milne} cor. 3.28 page 216.
 
 If $S$ is a minimal K3 surface, we  have $\chi(S) = 2, b_2(S) = 22$, hence in this case we get soon the estimate
 $ k = | Sing(X)| \leq 21.$
 
 Now, it follows that $S$ is a minimal K3 surface if the singular points are rational singularities (this means that  $\sR^1f_* (\hol_S)=0$): 
  because
  these are then rational double points (see \cite{artin}) and $K_S$ is a trivial divisor. However this does not need to hold
  in general. 
 
 If instead the singular points are not rational, it follows that $K_S$ is the opposite of an effective 
 exceptional divisor and  $h^1(\hol_S) >0$.  $ K_S^2$ is    negative, and $\chi(S)$ nonpositive; if both are negative ($h^1(\hol_S) \geq 2$) by Castelnuovo's theorem $S$ is ruled and
 possibly non minimal. Hence in this case we do not have an explicit upper bound for $b_2(S) = - K_S^2 + 12 \chi(S)$.
 \end{rem}

In order to see whether the minimal resolution $S$ is a K3 surface, we observe that $A_1$-singularities,
that is, double points $P$ for which 
the projectivized tangent conic  $Q$ is smooth, are such that the blow up of $P$ yields a resolution, containing a $(-2)$-curve as exceptional divisor.

If instead the conic $Q$ consist of two lines, $x_1 x_2=0$ (one says that $P$ is a biplanar double point), 
blowing up the point $P$ we obtain $X_1 \ra X$ 
with two generically reduced exceptional curves, meeting in one point $P_1$ which is the only possible singular point lying over $P$.

From the affine Taylor expansion of $F$ at $P$:
$$ (Taylor^a) \ F (x_1, x_2,x_3) = x_1 x_2  +   G (x) + B (x) ,$$

we derive that the local equation of $X_1$ at $P_1 = \{ u_1=u_2=x_3=0\}$ is: 
$$  \ F' (u_1, u_2,x_3) = u_1 u_2  +   x_3 G (u_1, u_2,1) + x_3^2 B (u_1, u_2,1) .$$
$P_1$ is smooth if and only if  $G$ contains the monomial $x_3^3$ with nonzero coefficient,
and otherwise $P_1$ is an $A_1$-singularity, and the blow up of $X_1$ is a local resolution of $X$, if   $B$ contains the monomial $x_3^4$ with nonzero coefficient.

Therefore we infer that $P$ is  a rational double point of one of the  types $ A_2, A_3$ unless $G,B$ belong to the ideal
$(x_1, x_2)$.

In the contrary case, $P_1$ is again a biplanar double point.

Using the next Lemma we find that the minimal resolution of these points is always a 
rational double point of   type $A_n$, for some $n$.

\begin{lemma}\label{A_n}
Let $P \in X = \{ F(x_1, x_2, x_3) = 0\}$ be a biplanar double point of an affine surface (that is, 
with tangent cone union of two different planes): then 

i)  its blow up is either smooth, or it contains a unique singular point, which is then
either again  a biplanar double point, or is an $A_1$-singularity (a double point with projectively smooth tangent cone).

ii) The minimal resolution is  of type $A_n$ with $n 
\geq 2$, that is, the exceptional divisor is a chain of $n$
curves $E_i \cong \PP^1$, with $ E_i^2 = -2$.
\end{lemma} 

 \begin{proof}
 From the affine Taylor expansion of $F$ at $P = \{ x=0\}$
$$ (Taylor^a) \ F (x_1, x_2,x_3) = x_1 x_2  +   G (x) + B (x) ,$$
where $G$ is homogeneous of degree $3$ and $B$ has order at least $4$,
exactly as before 
we derive that the local equation of the blow up $X_1$ at $P_1 = \{ u_1=u_2=x_3=0\}$ is: 
$$  \ F' (u_1, u_2,x_3) = u_1 u_2  +   x_3 G (u_1, u_2,1) + x_3^2 B (u_1, u_2,1) .$$

$P_1$ is smooth if and only if  $G$ contains the monomial $x_3^3$ with nonzero coefficient,
equivalently $G (0, 0,1) \neq 0$, 
else we have a double point with quadratic part $Q_1$ of the Taylor development equal to 
$$ Q_1 = u_1 u_2  +   x_3 G_1 (u_1, u_2,1) + x_3^2 B (0,0,1) .$$

Either $Q_1$ is the product of two linearly independent linear  forms,
or $Q_1$ is a smooth conic, and this if and only if  $B (0,0,1)  \neq 0$, equivalently,
  $B$ contains the monomial $x_3^4$ with nonzero coefficient.

i) is then proven and ii) follows easily by induction.

 \end{proof}

If instead we have a uniplanar double point, that is, an affine Taylor expansion of $F$ at $P$:
$$ \ F (x_1, x_2,x_3) = x_1 ^2   +   G (x) + B (x) ,$$
the blow up $X_1$ has a neighbourhood of the exceptional divisor 
contained in the union of two  affine pieces, one  with  equation 
$$  \ F' (u_1, u_2,x_3) = u_1 ^2  +   x_3 G (u_1, u_2,1) + x_3^2 B (u_1, u_2,1) ,$$

and another with equation
$$  \ F^*(u_1, x_2,u_3) = u_1 ^2  +   x_2 G (u_1, 1, u_3) + x_2^2 B (u_1,1,  u_3) .$$

We have that, in the first affine piece,  the intersection of the singular set of $X_1$ with the exceptional divisor $x_3=0$
equals 
$$ x_3 = u_1 = G (0, u_2,1)=0,$$
 respectively
$$ x_2 = u_1 = G (0, 1, u_3)=0$$
in the other afine piece.

Hence we  have a 1-dimensional singular set (then equal to $\{x_3= u_1=0\}$) for $X_1$
if and only if  

\bigskip

i) $x_1 | G$. 

\bigskip

Observe  that  for rational double points the resolution is obtained by a sequence of blow 
ups of (also infinitely near) points of multiplicity $2$. 

\bigskip

Let us consider the case where   $x_1$ does not divide $G$.  Then $G (0, u_2,1)$ is not identically zero and we get a singular point $P'$ with 
smooth tangent quadric for each solution of $ x_3 = u_1 = G (0, u_2,1)=0$ unless (possibly) if

\bigskip

ii) $G (0, u_2,1) = 0$ has a double root at $P'$, 

or

iii) $G (0, u_2,1) = 0$ has a triple root at $P'$.

Similarly for the other chart, so by looking at the roots of $G (0, u_2,u_3) = 0$ and possibly exchanging the roles of $x_2$ and $x_3$
we may reduce to the case where the unique root with multiplicity $\geq 2$ is in the first chart.

\bigskip
Changing variables both cases i) and ii) reduce to the condition  that $ G = x_1 q (x) + \la x_2^2 x_3$,  case 
iii) reduces to: $ G = x_1 q (x) +  x_2^3$.

\medskip

{\bf Example:}
The former case i)  ($\la = 0$) happens for instance for
$$X = \{ F = z^2 x_1^2 + z x_1^2 x_2 + x_1^3 x_2 + x_2^3 x_3 + x_3^3 x_1=0\},$$
whose only singular point is $P= \{x_1= x_2 = x_3 = 0\}$.

\bigskip

If $\la \neq 0$, we may assume without loss of generality that $\la = 1$, and at  the point $P' = \{x_3 = u_1 = u_2=0\}$
we have a  local equation 
$$  \ F' (u_1, u_2,x_3) = u_1 ^2  +   x_3  (u_1 q(u_1, u_2,1) + u_2^{2 + \de}) + x_3^2 B (u_1, u_2,1) ,$$
with $\de = 0 $ in case ii), $\de = 1 $ in case iii).

W get a point with  tangent cone  union of two distinct planes unless $q(0, 0 ,1)= 0$,
which means that $q(x)$ is in the ideal $(x_1, x_2 )$,
and implies that  

ii) : $ G(x) = x_1^2 L(x) + x_1 x_2 M(x_2, x_3) + x_2^2 x_3,$

iii) : $ G(x) = x_1^2 L(x) + x_1 x_2 M(x_2, x_3) + x_2^3.$

\bigskip

 We are allowed to change variables, but keeping $x_1$ fixed.

Write now $M(x_2, x_3) = m_2 x_2 + m_3 x_3.$ Then in case ii):

 $$ G(x) = x_1^2 L(x) + m_2 x_1 x_2^2 +  m_3 x_1 x_2 x_3 + x_2^2 x_3 =   x_1^2 L(x)  +  m_3 x_1 x_2 x_3 + x_2^2 (x_3+ m_2 x_1)$$
 and, setting $y_3 : = x_3+ m_2 x_1$,

$$ G(x) = x_1^2 L(x) + x_2^2 y_3 +  m_3 x_1 x_2 (y_3 + m_2 x_1) = x_1^2 L'(x_1, x_2, y_3) + x_2^2 y_3 + m_3 x_1 x_2 y_3.$$

Writing $  L'= a^2  y_3 + \hat{ L}(x_1, x_2)$,  we can rewrite 
$$G = y_3 [x_2^2 + a^2 x_1^2 + m_3 x_1 x_2] +   x_1^2 \hat{ L}(x_1, x_2) = y_3 y_2 \tilde{N} (x_1, y_2 ) +   x_1^2 \tilde{ L}(x_1, y_2)$$
(taking the roots of the quadratic polynomial  inside the square brackets).

Changing again variables we reach the simpler form 

$$ (ii)  \  G(x) = x_1^2 L(x_1, x_2) + x_2 x_3 N (x_1 , x_2).$$

In case iii),  we remain with  the form 
$$ (iii) \ G(x) = x_1^2 L(x) + x_1 x_2 M(x_2, x_3) + x_2^3.$$

\bigskip

We conclude recalling that   in the case  $ Q = x_1^2$ we can use II) of Proposition \ref{separable}, implying that
$ | Sing(X)| \leq 17$ in case i), $ | Sing(X)| \leq 19$ in cases ii),
 and  iii).

\bigskip

We reach then our final result

\begin{theo}\label{mt}
Let $X$ be  quartic surface $X = \{ F =0\} \subset \PP^3_K$, where $K$ is an algebraically closed field
of characteristic  $2$,
and suppose  that $ Sing (X)$ is a finite set. 

Then $ | Sing(X)| \leq 20$.

In fact:

(1) if $ | Sing(X)| \geq 20$, the only singularities of $X$ are rational double points and the minimal resolution
$S$ of $X$ is a minimal K3  surface; moreover $ | Sing(X)| \leq 21$.

(2) If it were $ | Sing(X)| =  21$, then the singularities would  be nodes (ordinary quadratic singularities), and $S$ 
would have  $\rho(S) = b_2(S) = 22$. This case does not exist.

(3) If $ | Sing(X)| = 20$,  then  the singularities are nodes (ordinary quadratic singularities), and $S$ 
is supersingular, this means that the  Picard number $\rho(S)=   b_2(S) = 22$.

\end{theo}

\begin{proof}
If $X$ has a point of multiplicity $4$, then $ | Sing(X)| =1$.

If $X$ has a point of multiplicity $3$, then by Proposition \ref{monoid} $ | Sing(X)| \leq 7$.

If $X$ has a point of multiplicity $2$ such that the projection with centre $P$ is inseparable, then $ | Sing(X)| \leq 16$ by Proposition \ref{inseparable}.

If $X$ has a point of multiplicity $2$ such that the projection with centre $P$ is separable, 
 the projective tangent conic at $P$ is a double line $Q = x_1^2$,
 either (by Lemma \ref{A_n} and the ensuing arguments) the minimal resolution of $P$ yields a rational double point,  or
i) or ii) hold and  $ | Sing(X)| \leq 19$ by Proposition \ref{separable}.

It follows that the minimal resolution $S$ of $X$
is a minimal K3 surface if $ | Sing(X)| \geq 20$.

Using remark \ref{K3}, we get    the estimate
 $ | Sing(X)| \leq  \rho(S) -1 \leq 21.$
 
(2)  If equality holds, then $ \rho(S) = 22$ and there are exactly $21$ irreducible exceptional curves $E_i$, with $E_i^2 = -2$,
 hence the singularities are nodes.
 
 As pointed out by a  referee,   the existence of  $21$ ordinary double points
and of a  divisor $H$ with $H^2=4$ which is orthogonal to the $21$ disjoint $(-2)$
curves leads to a contradiction.

 Because the discriminant of the rank $22$ Neron Severi lattice $\sN : = NS(S)$
has  absolute value  $  2^{2\s} $ \cite{artinSS} \cite{rudakov-shafarevich}, where $\s = \s(S)$ is the so-called Artin invariant,
while the  lattice $\sL$ spanned by $H$ and the $E_i$' s,
whose   index in  $\sN$  we denote $b$,  has  discriminant of  absolute value
$  2^{23} $.  We have a contradiction from $  2^{23} = b^2 2^{2m}$.

(3) Each singular point $P$ contributes $r(P)$ to the Picard number $\rho(S)$ if the local resolution
has $r(P)$ $(-2)$ curves.  Then  $\sum_P r(P) + 1 \leq \rho(S) \leq 22$.

Artin proved \cite{artinSS} that K3  surfaces with height of the formal Brauer group $h$ ($h \in \NN \setminus\{0\}  \cup \{ \infty \}$)
   satisfy $ \rho(S) \leq 22 - 2h$ if the formal Brauer group is $p$-divisible. 
   
   Artin observes that K3 surfaces $S$ with $ \rho(S) = 21$ do not exist, as he proves that 
   if $h = \infty$ and $S$ is elliptic, then the formal Brauer group is $p$-divisible,
   and moreover $S$ is elliptic once $\rho(S) \geq 5$ by Meyer's theorem (see \cite{serre}
   corollary 2, page 77).
   
   Artin predicted (modulo the conjecture that $h = \infty$ implies that $S$ is elliptic)
 that  $h = \infty \Leftrightarrow  \rho(S) = 22$.
 
 This equivalence follows from the Tate conjecture for K3 surfaces over finite fields, as explained in \cite{liedtke},
 discussion in section 4, especially theorem 4.8 and remark 4.9; the Tate conjecture was proven in $char=2$
 by Charles, Theorem 1.4 of \cite{charles}.

Hence  $\sum_P r(P) = 20$ or $21$. In the first case we have $20$ nodes, in the second case we would
have $19$ nodes and an $A_2$ singularity. 

In the second case  however the  lattice $\sL$ spanned by $H$ and the $E_i$' s has discriminant of  absolute value
$  2^{21} \cdot 3$, $S$ is a supersingular K3 surface, hence we get the usual contradiction 
$  2^{21} \cdot 3 = b^2 2^{2m}$.

\end{proof}

\begin{rem}
The K3  surfaces with  $\rho(S) = b_2(S) = 22$ are the so-called Shioda-supersingular K3 surfaces.

Shioda observes  \cite{shioda} that Kummer surfaces in characteristic $2$ have 
at most $4$ singular points, and proves that the Kummer surface associated to a product of elliptic
curves has $\rho(S) \leq 20$.

Rudakov and Shafarevich \cite{rudakov-shafarevich} described the supersingular K3 surfaces in characteristic $2$ according to their Artin invariant $\s$, which determines the intersection form on $Pic(S)$.

Dolgachev and Kondo \cite{dolga-kondo} constructed a supersingular K3 surface in characteristic $2$ with Artin invariant $1$
and with $21$ disjoint $-2$
curves, but  this surface does not  have  an embedding as a quartic surface with $21$ ordinary double points
(see also \cite{katsurakondo}), as we saw in the proof of theorem \ref{mt}.

In their case the orthogonal to the $21$ disjoint $-2$
curves is a divisor $H$ with $H^2=2$, yielding a realization as a double plane.

It seems unlikely that  a supersingular K3 surface can be birational to a nodal   quartic surface.

Matthias Sch\"utt proposed a strategy to exclude the case $\nu = 20$ (as we saw,  $S$ 
is then supersingular).

\end{rem}

\section{Symmetric Quartics}

One can try to see whether, as it happens for cubic surfaces or for quartics in characteristic $\neq 2$, one can construct quartics
with the maximum number of singular points using quartic surfaces admitting $\mathfrak S_4$-symmetry.

But the main result of this section is negative in this direction:

\begin{theo}\label{symmetric}
The special quartic surfaces admitting $\mathfrak S_4$-symmetry, with equation of the form

$$ F (x) : = F(a,x) : =  a_1 \s_1^4 + a_2 \s_1^2 \s_2 +  a_3 \s_1 \s_3 + a_4  \s_4  = 0,$$ 
(where $\s_i$ is as usual the $i$th elementary symmetric function)
have at least $6$  singular points (the  $\mathfrak S_4$-orbit of $(0,0,1,1)$) and either
  \begin{itemize}
  \item
  $6$ singular points, or
  \item
   $10$ singular points, the $4$ extra singular points being either the $\mathfrak S_4$-orbit of the point $(0,0,0,1)$, or the $\mathfrak S_4$-orbit of a point $(1,1,1,b)$, or
   \item
 infinitely many  singular points.
 \end{itemize}

\end{theo}
 
 We shall prove the theorem through a sequence of auxiliary and more precise results.

Let the  quartic $X$ have an equation of the form

$$ F (x) : = F(a,x) : =  a_1 \s_1^4 + a_2 \s_1^2 \s_2 +  a_3 \s_1 \s_3 + a_4  \s_4  = 0:$$ 
since we want $X$ not to be reducible, we may soon assume that $a_4 \neq 0$.

A calculation of the partial derivatives yields:
$$ F_i : = \frac{\partial F}{\partial x_i} = a_2 \s_1^2 ( \s_1 + x_i) + a_3 (  \s_3 + \s_1 (\s_2 + \s_1 x_i + x_i^2)) + a_4  \frac{\s_4}{x_i}.$$

The singular set of $X$ is the set 
$$Sing(X) : = \{ x | F(x) = F_i (x) =0 , \ \forall 
1 \leq i \leq 4\},$$
which is clearly a union of $\mathfrak S_4$ orbits.

\begin{rem}\label{sing}
If a singular point of $X$ has two coordinates equal to zero,  say $x_1=x_2=0$, then $a_4  \frac{\s_4}{x_i}$ vanishes for all $i$,
and for $i=1,2$  we get that  the equation  $ a_2 \s_1^3  +  a_3 (  \s_3 + \s_1 \s_2)=0$ must be satisfied by the singular point.

Once this equation is satisfied, then for $i=3,4$ 
$$F_i=0 \Leftrightarrow a_2 \s_1^2  x_i  + a_3 (   \s_1 ( \s_1 x_i + x_i^2)) =0,$$
which vanishes exactly for $\s_1=0$, or for  $x_i=0$, or for 
$$   a_3 x_i + ( a_3 + a_2)    \s_1 =0 \Rightarrow  a_3 ( x_3 +  x_4) =0.$$

For $\s_1=0$ we get the point $(0,0,1,1)$ (and its $\mathfrak S_4$-orbit).

For $x_3=0$, $\s_1=1$, hence the only possibility for $x_4$ is $a_2 x_4=0$, and we get the  point
$(0,0,0,1)$ (and its $\mathfrak S_4$-orbit) if $a_2=0$.

If $a_3=0$, either $\s_1= 0$, or  $a_2=0$, but then if the point lies in $X$ we must have again $\s_1=0$
or $a_1=0$.

\end{rem}

\begin{lemma}
The quartic $X_a : = \{ x | F(a,x)=0\}$ has the property that $Sing(X)$ contains the $\mathfrak S_4$-orbit
of the point $(0,0,1,1)$, and contains the $\mathfrak S_4$-orbit
of the point $(0,0,0,1)$ if and only if $a_1= a_2= 0$.
\end{lemma} 
\begin{proof}
By remark \ref{sing} it suffices to observe that $(0,0,1,1)$ belongs to $X_a$, for each choice of $ a = (a_1,a_2, a_3, a_4)$.

We calculate more generally, for later use:
$$ (Sym) \  \s_1 (b,c,1,1) = \s_3 (b,c,1,1) = b+c, \ \s_2 (b,c,1,1) = 1 + bc, \  \s_4 (b,c,1,1) = bc.$$

For $b=c=0$ we get that all $\s_i$ vanish except $\s_2 =1$, hence we are done.

For the point $(0,0,0,1)$   all $\s_i$ vanish except $\s_1 =1$, hence this point is in $X_a$ if and only if  $a_1 = 0$,
and we apply again remark \ref{sing} to infer that $a_1= a_2=0$.

\end{proof}

\begin{lemma}
The quartic $X_a : = \{ x | F(a,x)=0\}$, where we assume that $a_4=1$,  has the property that $Sing(X)$ contains the $\mathfrak S_4$-orbit
of a point $(0,1,x_3,x_4)$, with $x_3, x_4 \neq0$,  if and only if  this orbit is  the 
 $\mathfrak S_4$-orbit of $(0,1,1,z )$ and  $$a_1 z^2 = (a_2 + a_3),  \ a_2 z^2 =  1 , \ (  a_2 z +1) =  a_3 (z +  1).$$
 In particular, we must have $a_2 \neq0$, 
 and $z=1$ iff $a_2 = 1, a_1 = 1  + a_3$, while for $z\neq 1$  
 $$a_2 = \frac{ 1} {  z^2}  ,  a_3 = \frac{ 1} {  z} , \ a_1 = \frac{ 1+z} {  z^4}.$$
 
\end{lemma}\label{onezero}
\begin{proof}
For $j=1$ we know that 
$$a_2 \s_1^3  + a_3 (  \s_3 + \s_1 \s_2 )) + a_4  \s_3=0,$$
hence for $j=2,3,4$
$$a_2 \s_1^2   x_i + a_3 (   \s_1 ( \s_1 x_i + x_i^2)) + a_4  \s_3=0.$$
This is an equation of degree $2$, hence it has at most two roots, not being identically zero.

Hence we may assume that our point is of the form $(0,1,1,z )$, and then $\s_1 = z, \ \s_2 = 1, \ \s_3 = z, 
\s_4 = 0$.

Since we want $z \neq0$, the conditions that the point is in $X$, plus that we have a singular point (hence $1,z$ are roots of the quadratic equation) are:
$$a_1 z^2 = (a_2 + a_3),  \ a_2 z^2 =  a_4 , \ (  a_2 + a_3 ) z = a_3 +  a_4.$$

\end{proof}

Consider next a point of the form $(b,c,1,1)$,  for $ b \neq c $,$ b,c \neq 1$.
Its orbit consists of $12$ points.

\begin{prop}\label{inf}
The quartic 
$$ X_a = \{ x| a_1 \s_1^4 + a_2 \s_1^2 \s_2 +  a_3 \s_1 \s_3 + a_4  \s_4  = 0\} $$
contains the $\mathfrak S_4$-orbit of the point $(1,1,0,0)$, and 
the $\mathfrak S_4$-orbit of the point $(b,c,1,1)$, if and only if, setting $ w: = (b + c) \neq 0$,
the coefficients are proportional to 
$$ a_3 = 1, \ a_2 w = 1, \ a_4 = w, \ a_1 w^3 = 1 + w.$$
In particular, if these conditions are satisfied, then $X$ has infinitely many singular points.

\end{prop}
\begin{proof}
By the previous Lemma  $ Sing(X)$ contains the $\mathfrak S_4$-orbit
of the point $(0,0,1,1)$.

 We need only to see that the point $(b,c,1,1)$ is a singular point of $X_a$ for the above choices of $a$,
 depending on the two parameters $b,c$.
 
 First of all we get a point of $X_a$ since by formula {\em (Sym)} 
 $$(1+b+c) (b+c)^{-3}  \s_1^4 + (b+c)^{-1} \s_1^2 \s_2 +   \s_1 \s_3 + (b+c)   \s_4  = 0$$
 at the point  $(b,c,1,1)$, since
$$(1+b+c) (b+c)  + (b+c)  (1+bc)  +   (b+c)^2 + (b+c)   bc  = (b+c) ( b+ c + bc + (b+c) + bc) = 0.$$ 

For the partial derivatives, the condition for the singular points with all coordinates different from zero
boils down, for the given point, to

 $$ \F_i x_i = x_i (b+c)  ( b+ c + x_i) +  x_i (  b+c ) ( bc + (b+c) x_i + x_i^2)) + (b+c)   bc =0 $$
 which is satisfied if we show that $1,b,c$ are roots of the equation (we divide by $(b + c)$
 $$  z^3 + z^2  (b+c+1) + z   ( b+ c + bc ) +   bc =0 .$$
 We easily find that $1$ is a root, and then the other two roots are shown to be $b,c$ since their sum equals $b+c$, 
 and their product equals $bc$.
 
 Conversely, if $1,b,c$ are roots of the equation
 $$x_i  \F_i =   a_2 (b+c)^2 ( (b+c)  z + z^2 ) + a_3  (   (b+c)  bc  z + (b+c) ( (b+c) z^2 + z^3)) + a_4  bc = 0,$$
 that is, of 
 $$z^3 a_3 (b+c)  + z^2 (  a_2 + a_3)(b+c)^2   + z  ( a_2 (b+c)^2  + a_3 bc  )(b+c)   + a_4  bc = 0,$$
 then the equation has degree $3$, and we may normalize $a_3=1$ and divide by $(b+c)$, 
 obtaining
 $$z^3   + z^2 (  a_2 + 1)(b+c)   + z  ( a_2 (b+c)^2  +  bc  )   + a_4 \frac{ bc}{b+c} = 0,$$
 and looking at the sum of the roots, we get
 $$ 1 + b + c = (  a_2 + 1)(b+c) \Leftrightarrow   a_2 (b+c)= 1,$$ 
 moreover looking at the product of the roots we see that
 $$ a_4 = (b+c),$$ finally the condition 
 $  a_2 (b+c)^2  +  bc = b+c + bc $ is automatical.

 To finish the proof, we observe that 
 since the equation depends on $w$, we get infinitely many singular points varying $b,c$ with $b + c= w$.

\end{proof}

\begin{cor}
If $Sing(X)$ is finite, then $X$ has at most $10$ singular points.

\end{cor}
\begin{proof}

 The singular points with two coordinates equal zero are only in the orbit of $(0,0,1,1)$, 
 while we have singular points with three coordinates equal to  zero iff $a_1=a_2=0$.
 
 If we have a singular point with just one coordinate equal to zero, then we may apply lemma \ref{onezero}
 and find that the coefficients are as described in proposition \ref{inf}, hence we have infinitely many singular points.

Concerning  the singular points with $x_i \neq 0 \ \forall 1=1,2,3,4$ the coordinates are solutions of the degree $3$ equation:

$$x_i  \F_i =   a_2 \s_1^2 ( \s_1 x_i + x_i^2 ) + a_3  ( (  \s_3 + \s_1 \s_2) x_i + \s_1 ( \s_1 x_i^2 + x_i^3)) + a_4  \s_4 = 0.$$

Therefore these coordinates cannot be all different, unless the equation is identically zero, which means that 

$a_3 \s_1 = 0, (a_2 + a_3) \s_1^2 = 0, a_2 \s_1^3 + a_3 (  \s_3 + \s_1 \s_2)  = 0 ,  \s_4=0 
$, hence $\s_4 =0$ and the $4$ coordinates are not all different from $0$.

If the coordinates take three distinct values, we obtain the situation of the previous proposition \ref{inf}, 
hence infinitely many singular points.

Therefore remains only the possibility of just two values. 

If the point is $(1,1,b,b)$ then $\s_1=\s_3 = 0$, $\s_2 = 1 + b^2$, $\s_4 = b^2$.

But then the equation of $X$ implies $\s_1=\s_3 = \s_4 = 0$, hence we get $b=0$, the usual orbit.

If instead the point is $(1,1,1,b)$ it follows that $\s_1= 1 + b, \s_4 = b, \s_2 =  1+b, \s_3 = 1+b.$

It cannot be $b=1$, since $\s_1=0 \rightarrow \s_4=0$.

The condition that $1,b$ are roots of the cubic equation is equivalent to 
$$ a_2 (1+b)^2 = a_4,$$
and since $(1+b) \neq 0, a_4 \neq 0$, it cannot be $a_2=0$.

Since $a_2 \neq 0$,  $b$ is uniquely determined by this equation and we have at most $10$ singular points
also in this case.

\end{proof}

\begin{cor}
In the pencil of quartics $$ X_c: = \{ c \s_1 \s_3 + \s_4 =0\}$$
the quartic has always $10$ singular points except for $c=0$.

The singularities are nodes ($A_1$-singularities).

\end{cor}
\begin{proof}
By remark \ref{sing} the only singular points with at least two coordinates equal to zero
are just the orbits of $(0,0,1,1)$ and $(0,0,0,1)$ if $a_3 = c\neq 0$.

Points with just one coordinate equal to zero are excluded by lemma \ref{onezero},
while singular points with nonzero coordinates taking three values are excluded by
proposition \ref{inf}. 

For points of type $(1,1,b,b)$ $\s_1=0$  hence they cannot lie in $ X_c$, since $a_4=1$.

For points of type $(1,1,1,b)$, $b\neq 1$, they lie in $ X_c$ iff $ c (1+b)2 = b$,
but as we saw the condition that we have a singular point boils down to
$$ a_2 (1+b)^2 = a_4=1,$$ impossible since $a_2=0$.

For the last assertion, at the point $(0,0,1,1)$ we have $x_3=1$ and local coordinates $x_1, x_2, \s_1$: 
then the quadratic part of the equation is
$$ x_1 x_2 + \s_1 (x_1 +x_2)$$
 and we have a node.

Likewise, at the point $(0,0,0,1)$ we have $x_4=1$ and local coordinates $x_1, x_2, x_3$:
then the quadratic part of the equation is
$$ \s_3 = x_1 x_2 +x_1 x_3 + x_2 x_3 $$
and again we have a node.

\end{proof}

\bigskip

\bigskip

\section{Appendix}
Here another construction of  quartics with $14$ singular points,  due to Matthias Sch\"utt.

Consider the quartic $X$ of equation

$$X : = \{  F (w,x_1,x_2,x_3) : = w^4 + w^2 x_1^2 + B(x) = 0\}.$$

The singular points are the solutions of 
$$ \nabla B(x)= 0 , w^4 + w^2 x_1^2 + B(x) = 0.$$

For each $x$, the polynomial $w^4 + w^2 x_1^2 + B(x)$ is the square of a quadratic polynomial in $w$,
which is separable if $x_1\neq 0$.

Let $$\sC : = \{ x \in \PP^2 | \nabla B(x)= 0 \}.$$

Hence $ |Sing(X)| = 2 | \sC|$ provided $ \sC \cap \{x_1=0\} = \emptyset,$
a condition which can be realized for the choice of a general  linear form
once we find a quartic $B$ with $\sC$ finite.

The first choice is to take $B$ the product of $4$ general linear forms, as in Step IV)
of proposition \ref{inseparable}, where $\sC$ consists of $7$ points,
hence we get $X$ with $14$ singular points.

Also for $\{ B = 0\}$ smooth we get $\sC$ with $7$ points. To avoid using advanced algebraic geometry,
let us consider the Klein quartic
$$ B : = x_1^3 x_2 + x_2^3 x_3 + x_3^3 x_1.$$
Then $$ B_1 = x_1^2 x_2 + x_3^3 ,  \  B_2 = x_2^2 x_3 + x_1^3 , \ B_3 = x_3^2 x_1 + x_2^3 .$$

Using the first equation and taking the cubes of $x_2^2 x_3 = x_1^3$, we find (since $x_i \neq0$ for the points of $\sC$)
that $ y = \e x$, with $\e^7=1$, and $ z = \e^5 x$. Hence we get the seven points $ \{ (1, \e, \e^5)| \e^7 =1\}$.

Recall that the set $\sC$ has always cardinality at most $7$ in view of Lemma \ref{quartic},
so this construction leads to no more than $14$ singular points.

\bigskip

{\bf Acknowledgement:} I would like to thank Stephen Coughlan,  the referee of \cite{kummers},
  and especially Matthias Sch\"utt for interesting remarks, helpful comments and encouragement.

\end{document}